\documentclass[12pt]{article}

\usepackage{amsmath,amssymb,amsthm,mathrsfs,amsfonts,dsfont,txfonts}

\usepackage[dvipsnames]{xcolor}

\usepackage{MnSymbol}

\newtheorem{theorem}{Theorem}[section]
\newtheorem{corollary}{Corollary}[section]

\newtheorem{claim}{Claim}[section]

\begin{document}
\thispagestyle{empty}
\begin{center}
{\bf \Large Mappings preserving a family of sum of triple products on $\ast $-algebras}\\
\vspace{.2in}
{\bf Jo\~{a}o Carlos da Motta Ferreira\\ and \\Maria das Gra\c{c}as Bruno Marietto}\\
\vspace{.2in}
{\it Center for Mathematics, Computing and Cognition,\\
Federal University of ABC,\\
Avenida dos Estados, 5001, \\
09210-580, Santo Andr\'e, Brazil.}\\
e-mail: joao.cmferreira@ufabc.edu.br, graca.marietto@ufabc.edu.br\\
\end{center}

\begin{abstract} Let $\mathcal{A}$ and $\mathcal{B}$ be two unital complex $\ast $-algebras such that $\mathcal{A}$ has a nontrivial projection.  In this paper, we study the structure of bijective mappings $\Phi :\mathcal{A}\rightarrow \mathcal{B}$ preserving sum of triple products $\alpha _{1} ab^{*}c+\alpha _{2} acb^{*}+\alpha _{3} b^{*}ac +\alpha _{4} cab^{*}+\alpha _{5} b^{*}ca+\alpha _{6} cb^{*}a,$ where the scalars $\{\alpha _{k}\}_{k=1}^{6}$ are complex numbers satisfying some conditions. Applications of obtained results are given.
\end{abstract}

\section{Introduction}
Let $\mathcal{A}$ and $\mathcal{B}$ be two complex $\ast $-algebras. For $a,b\in \mathcal{A}$ (resp., $a,b\in \mathcal{B}$) denote by $a\filledsquare _{\eta} b=a^{*}b+\eta ba^{*}$ and $a\circ _{\eta} b=ab+\nu ba,$ where $\eta ,\nu $ are nonzero complex numbers. We say that a mapping $\Phi:\mathcal{A}\rightarrow \mathcal{B}$ {\it preserves triple product $a\filledsquare _{\eta }b\filledsquare _{\nu }c$} (resp., {\it preserves mixed product $a\filledsquare _{\eta }b\circ _{\nu }c$}), where $a\filledsquare _{\eta }b\filledsquare _{\nu }c=(a\filledsquare _{\eta }b)\filledsquare _{\nu }c$ (resp., $a\filledsquare _{\eta }b\circ _{\nu }c=(a\filledsquare _{\eta }b)\circ _{\nu }c$), if 
{\allowdisplaybreaks\begin{align*}\allowdisplaybreaks
&\Phi (a\filledsquare _{\eta }b\filledsquare _{\nu }c)=\Phi (a)\filledsquare _{\eta }\Phi (b)\filledsquare _{\nu }\Phi (c)\nonumber \\
& \hspace{2.0cm} (\textrm{resp.,} \,\, \Phi (a\filledsquare _{\eta }b\circ _{\nu }c)=\Phi (a)\filledsquare _{\eta }\Phi (b)\circ _{\nu }\Phi (c)),
\end{align*}}
for all elements $a,b,c\in \mathcal{A}.$

Let $\mathcal{A}$ and $\mathcal{B}$ be two complex $\ast $-algebras, $\{\alpha _{k}\}_{k=1}^{6}$ complex numbers and $\Phi :\mathcal{A}\rightarrow \mathcal{B}$ a mapping. We say that $\Phi $ {\it preserves sum of triple products $\alpha _{1} ab^{*}c+\alpha _{2} acb^{*}+\alpha _{3} b^{*}ac +\alpha _{4} cab^{*}+\alpha _{5} b^{*}ca+\alpha _{6} cb^{*}a$} if
{\allowdisplaybreaks\begin{align}\allowdisplaybreaks\label{fundident}
&\Phi (\alpha _{1} ab^{*}c+\alpha _{2} acb^{*}+\alpha _{3} b^{*}ac +\alpha _{4} cab^{*}+\alpha _{5} b^{*}ca+\alpha _{6} cb^{*}a)\nonumber \\
&=\alpha _{1}\Phi (a)\Phi (b)^{*}\Phi (c)+\alpha _{2}\Phi (a)\Phi (c)\Phi (b)^{*}+\alpha _{3}\Phi (b)^{*}\Phi (a)\Phi (c)\nonumber \\
&+\alpha _{4}\Phi (c)\Phi (a)\Phi (b)^{*}+\alpha _{5}\Phi (b)^{*}\Phi (c)\Phi (a)+\alpha _{6}\Phi (c)\Phi (b)^{*}\Phi (a),
\end{align}}
for all elements $a,b,c\in \mathcal{A}.$ 

In recent years, there has been considerable interest in the study of mappings preserving different types of products, triple products or mixed products on $\ast $-algebras (for example, see the works \cite{Darvish}, \cite{Liu}, \cite{Taghavi}, \cite{Zhao1}, \cite{Zhao2} and the references therein). In particular, Darvish et al. \cite{Darvish} studied the structure of the mappings preserving product $a\filledsquare _{\eta} b$ on $C^{*}$-algebras, Liu and Ji \cite{Liu} studied the structure of the mappings preserving product $a\filledsquare _{1} b$ on factor von Neumann algebras and Taghavi et al. \cite{Taghavi} studied the structure of the mappings preserving triple product $a\filledsquare _{1} b\filledsquare _{1} c$ on $\ast $-algebras. Note that mappings preserving triple product $a\filledsquare _{1} b\filledsquare _{1} c$ satisfy (\ref{fundident}), for convenient scalars $\alpha _{k}$ $(k=1,2,\cdots ,6).$ Based on these facts, in this paper, we study the mappings that preserves sum of triple products $\alpha _{1} ab^{*}c+\alpha _{2} acb^{*}+\alpha _{3} b^{*}ac +\alpha _{4} cab^{*}+\alpha _{5} b^{*}ca+\alpha _{6} cb^{*}a$ on $\ast $-algebras. Applications of obtained results are given to mappings preserving triple product $a\filledsquare _{\eta }b\filledsquare _{\nu }c$ and preserving mixed product $a\filledsquare _{\eta }b\circ _{\nu }c.$

\section{The statement of the main results}
Two main results are given in this paper. The first read as follows.

\begin{theorem}\label{thm21} Let $\{\alpha _{k}\}_{k=1}^{6}$ be complex numbers satisfying the condition $\sum _{k=1}^{6} \alpha _{k} \neq 0,$ $\mathcal{A}$ and $\mathcal{B}$ two unital complex $\ast $-algebras  with $1_{\mathcal{A}}$ and $1_{\mathcal{B}}$ their multiplicative identities, respectively, and such that $\mathcal{A}$ is prime and has a nontrivial projection. Then every bijective mapping $\Phi :\mathcal{A}\rightarrow \mathcal{B}$ preserving sum of triple products $\alpha _{1} ab^{*}c+\alpha _{2} acb^{*}+\alpha _{3} b^{*}ac +\alpha _{4} cab^{*}+\alpha _{5} b^{*}ca+\alpha _{6} cb^{*}a$ is additive. In addition, (i) if $\Phi (1_{\mathcal{A}})$ is a projection, then $\Phi $ is a $\ast $-Jordan ring isomorphism and (ii) if $\mathcal{B}$ is prime and $\phi (1_{\mathcal{A}})$ is a projection of $\mathcal{B},$ then $\Phi $ is either a $\ast $-ring isomorphism or an $\ast $-ring anti-isomorphism.
\end{theorem}

We organize the proof of Theorem \ref{thm21} in a series of Claims. The following three Claims will be used throughout this paper whose proofs are simple and therefore omitted here.

\begin{claim}\label{c21} If $\Phi $ preserves sum of triple products $\alpha _{1} ab^{*}c+\alpha _{2} acb^{*}+\alpha _{3} b^{*}ac +\alpha _{4} cab^{*}+\alpha _{5} b^{*}ca+\alpha _{6} cb^{*}a,$ then it also preserves sum of triple products $\alpha _{6} ab^{*}c+\alpha _{4} acb^{*}+\alpha _{5} b^{*}ac +\alpha _{2} cab^{*}+\alpha _{3} b^{*}ca+\alpha _{1} cb^{*}a.$
\end{claim}

\begin{claim}\label{c22} Let $a,b,c\in \mathcal{A}$ such that $\Phi (c)=\Phi (a)+\Phi (b)$. Then the following hold:
{\allowdisplaybreaks\begin{align*}\allowdisplaybreaks
(i)&\> \Phi (\alpha _{1} cs^{*}t+\alpha _{2} cts^{*}+\alpha _{3} s^{*}ct +\alpha _{4} tcs^{*}+\alpha _{5} s^{*}tc+\alpha _{6} ts^{*}c)\\
&=\Phi (\alpha _{1} as^{*}t+\alpha _{2} ats^{*}+\alpha _{3} s^{*}at +\alpha _{4} tas^{*}+\alpha _{5} s^{*}ta+\alpha _{6} ts^{*}a)\\
&+\Phi (\alpha _{1} bs^{*}t+\alpha _{2} bts^{*}+\alpha _{3} s^{*}bt +\alpha _{4} tbs^{*}+\alpha _{5} s^{*}tb+\alpha _{6} ts^{*}b),\\
(ii)&\> \Phi (\alpha _{1} st^{*}c+\alpha _{2} sct^{*}+\alpha _{3} t^{*}sc +\alpha _{4} cst^{*}+\alpha _{5} t^{*}cs+\alpha _{6} ct^{*}s)\\
&=\Phi (\alpha _{1} st^{*}a+\alpha _{2} sat^{*}+\alpha _{3} t^{*}sa +\alpha _{4} ast^{*}+\alpha _{5} t^{*}as+\alpha _{6} at^{*}s)\\
&+\Phi (\alpha _{1} st^{*}b+\alpha _{2} sbt^{*}+\alpha _{3} t^{*}sb +\alpha _{4} bst^{*}+\alpha _{5} t^{*}bs+\alpha _{6} bt^{*}s),
\end{align*}}
for all elements $s,t\in \mathcal{A}.$
\end{claim}

\begin{claim}\label{c23} $\Phi (0)=0.$
\end{claim}

The following well known result will be used throughout this paper: Let $p_{1}$ be an arbitrary nontrivial projection of $\mathcal{A}$ and write $p_{2}=1_{\mathcal{A}}-p_{1}.$ Then $\mathcal{A}$ has a Peirce decomposition $\mathcal{A}=\mathcal{A}_{11}\oplus \mathcal{A}_{12}\oplus \mathcal{A}_{21}\oplus \mathcal{A}_{22},$ where $\mathcal{A}_{ij}=p_{i}\mathcal{A}p_{j}$ $(i,j=1,2) ,$ satisfying the following multiplicative relations: $\mathcal{A}_{ij}\mathcal{A}_{kl}\subseteq \delta _{jk} \mathcal{A}_{il},$ where $\delta _{jk}$ is the {\it Kronecker delta function}.

\begin{claim}\label{c24} For arbitrary elements $a_{ii}\in \mathcal{A}_{ii},$ $b_{ij}\in \mathcal{A}_{ij}$ and $c_{ji}\in \mathcal{A}_{ji}$ $(i\neq j;i,j=1,2)$ the following hold: (i) $\Phi (a_{ii}+b_{ij})=\Phi (a_{ii})+\Phi (b_{ij})$ and (ii) $\Phi (a_{ii}+c_{ji})=\Phi (a_{ii})+\Phi (c_{ji}).$
\end{claim}
\begin{proof} By assumption on  $\Phi ,$ there exists $f=f_{ii}+f_{ij}+f_{ji}+f_{jj}\in \mathcal{A}$ $(i\neq j;i,j=1,2)$ such that $\Phi (f)=\Phi (a_{ii}) + \Phi (b_{ij})$. By Claims \ref{c22}(i) and \ref{c23}, we have 
{\allowdisplaybreaks\begin{align*}\allowdisplaybreaks
&\Phi (\alpha _{1} f1_{\mathcal{A}}^{*}p_{j}+\alpha _{2} fp_{j}1_{\mathcal{A}}^{*}+\alpha _{3} 1_{\mathcal{A}}^{*}fp_{j} +\alpha _{4} p_{j}f1_{\mathcal{A}}^{*}+\alpha _{5} 1_{\mathcal{A}}^{*}p_{j}f+\alpha _{6} p_{j}1_{\mathcal{A}}^{*}f)\\
&=\Phi (\alpha _{1} a_{ii}1_{\mathcal{A}}^{*}p_{j}+\alpha _{2} a_{ii}p_{j}1_{\mathcal{A}}^{*}+\alpha _{3} 1_{\mathcal{A}}^{*}a_{ii}p_{j} +\alpha _{4} p_{j}a_{ii}1_{\mathcal{A}}^{*}+\alpha _{5} 1_{\mathcal{A}}^{*}p_{j}a_{ii}+\alpha _{6} p_{j}1_{\mathcal{A}}^{*}a_{ii})\\
&+\Phi (\alpha _{1} b_{ij}1_{\mathcal{A}}^{*}p_{j}+\alpha _{2} b_{ij}p_{j}1_{\mathcal{A}}^{*}+\alpha _{3} 1_{\mathcal{A}}^{*}b_{ij}p_{j}+\alpha _{4} p_{j}b_{ij}1_{\mathcal{A}}^{*}+\alpha _{5} 1_{\mathcal{A}}^{*}p_{j}b_{ij}+\alpha _{6} p_{j}1_{\mathcal{A}}^{*}b_{ij})\\
&=\Phi ((\alpha _{1} +\alpha _{2} +\alpha _{3})b_{ij}).
\end{align*}}
This shows that $\alpha _{1} f1_{\mathcal{A}}^{*}p_{j}+\alpha _{2} fp_{j}1_{\mathcal{A}}^{*}+\alpha _{3} 1_{\mathcal{A}}^{*}fp_{j} +\alpha _{4} p_{j}f1_{\mathcal{A}}^{*}+\alpha _{5} 1_{\mathcal{A}}^{*}p_{j}f+\alpha _{6} p_{j}1_{\mathcal{A}}^{*}f=(\alpha _{1} +\alpha _{2} +\alpha _{3})b_{ij}$ which leads to $(\alpha _{1}+\alpha _{2}+\alpha _{3}) f_{ij}+(\alpha _{4}+\alpha _{5}+\alpha _{6}) f_{ji}+(\sum _{k=1}^{6} \alpha _{k})f_{jj}=(\alpha _{1}+\alpha _{2}+\alpha _{3}) b_{ij}.$ As a result of this identity, we deduce that $(\alpha _{6}+\alpha _{4}+\alpha _{5}) f_{ij}+(\alpha _{2}+\alpha _{3}+\alpha _{1}) f_{ji}+(\sum _{k=1}^{6} \alpha _{k})f_{jj}=(\alpha _{6}+\alpha _{4}+\alpha _{5}) b_{ij},$ in view of Claim \ref{c21}. Thus, by adding the two last identities we end up getting $(\sum _{k=1}^{6} \alpha _{k})(f_{ij}+f_{ji}+2f_{jj})=(\sum _{k=1}^{6} \alpha _{k})b_{ij}$ which allows the conclusion that $f_{ij}=b_{ij},$ $f_{ji}=0$ and $f_{jj}=0.$ Next, for an arbitrary element $t_{ij}\in \mathcal{A}_{ij},$ we have 
{\allowdisplaybreaks\begin{align*}\allowdisplaybreaks
&\Phi  (\alpha _{1} f1_{\mathcal{A}}^{*}t_{ij}+\alpha _{2} ft_{ij}1_{\mathcal{A}}^{*}+\alpha _{3} 1_{\mathcal{A}}^{*}ft_{ij} +\alpha _{4} t_{ij}f1_{\mathcal{A}}^{*}+\alpha _{5} 1_{\mathcal{A}}^{*}t_{ij}f+\alpha _{6} t_{ij}1_{\mathcal{A}}^{*}f)\\
&=\Phi  (\alpha _{1} a_{ii}1_{\mathcal{A}}^{*}t_{ij}+\alpha _{2} a_{ii}t_{ij}1_{\mathcal{A}}^{*}+\alpha _{3} 1_{\mathcal{A}}^{*}a_{ii}t_{ij}+\alpha _{4} t_{ij}a_{ii}1_{\mathcal{A}}^{*}+\alpha _{5} 1_{\mathcal{A}}^{*}t_{ij}a_{ii}+\alpha _{6} t_{ij}1_{\mathcal{A}}^{*}a_{ii})\\
&+\Phi  (\alpha _{1} b_{ij}1_{\mathcal{A}}^{*}t_{ij}+\alpha _{2} b_{ij}t_{ij}1_{\mathcal{A}}^{*}+\alpha _{3} 1_{\mathcal{A}}^{*}b_{ij}t_{ij}+\alpha _{4} t_{ij}b_{ij}1_{\mathcal{A}}^{*}+\alpha _{5} 1_{\mathcal{A}}^{*}t_{ij}b_{ij}+\alpha _{6} t_{ij}1_{\mathcal{A}}^{*}b_{ij})\\
&=\Phi ((\alpha _{1} +\alpha _{2} +\alpha _{3})a_{ii}t_{ij})
\end{align*}}
which implies that $\alpha _{1} f1_{\mathcal{A}}^{*}t_{ij}+\alpha _{2} ft_{ij}1_{\mathcal{A}}^{*}+\alpha _{3} 1_{\mathcal{A}}^{*}ft_{ij} +\alpha _{4} t_{ij}f1_{\mathcal{A}}^{*}+\alpha _{5} 1_{\mathcal{A}}^{*}t_{ij}f+\alpha _{6} t_{ij}1_{\mathcal{A}}^{*}f=(\alpha _{1} +\alpha _{2} +\alpha _{3})a_{ii}t_{ij}.$ As a consequence of this identity we get $(\alpha _{1}+\alpha _{2}+\alpha _{3})f_{ii}t_{ij}=(\alpha _{1}+\alpha _{2}+\alpha _{3})a_{ii}t_{ij}$ which allows to deduce that $(\alpha _{6}+\alpha _{4}+\alpha _{5})f_{ii}t_{ij}=(\alpha _{6}+\alpha _{4}+\alpha _{5})a_{ii}t_{ij},$ in view again of Claim \ref{c21}. Adding the last two results we obtain $(\sum _{k=1}^{6} \alpha _{k})f_{ii}t_{ij}=(\sum _{k=1}^{6} \alpha _{k})a_{ii}t_{ij}$ which results in  $f_{ii}t_{ij}=a_{ii}t_{ij}.$ Therefore $f_{ii}=a_{ii}.$

By an entirely similar reasoning, we prove the case (ii).
\end{proof}

\begin{claim}\label{c25} For arbitrary elements $b_{ij}\in \mathcal{A}_{ij}$ and $c_{ji}\in \mathcal{A}_{ji}$ $(i\neq j; i,j=1,2)$ the following holds $\Phi (b_{ij}+c_{ji})=\Phi (b_{ij})+\Phi (c_{ji}).$
\end{claim}
\begin{proof} By our hypotheses, there is $f=f_{ii}+f_{ij}+f_{ji}+f_{jj}\in \mathcal{A}$ $(i\neq j;i,j=1,2)$ satisfying $\Phi (f)=\Phi (b_{ij}) + \Phi (c_{ji})$.  We have by Claim \ref{c22}(i) that 
{\allowdisplaybreaks\begin{align*}\allowdisplaybreaks
&\Phi  (\alpha _{1} fp_{i}^{*}p_{i}+\alpha _{2} fp_{i}p_{i}^{*}+\alpha _{3} p_{i}^{*}fp_{i}+\alpha _{4} p_{i}fp_{i}^{*}+\alpha _{5} p_{i}^{*}p_{i}f+\alpha _{6} p_{i}p_{i}^{*}f)\\
&=\Phi  (\alpha _{1} b_{ij}p_{i}^{*}p_{i}+\alpha _{2} b_{ij}p_{i}p_{i}^{*}+\alpha _{3} p_{i}^{*}b_{ij}p_{i}+\alpha _{4} p_{i}b_{ij}p_{i}^{*}+\alpha _{5} p_{i}^{*}p_{i}b_{ij}+\alpha _{6} p_{i}p_{i}^{*}b_{ij})\\
&+\Phi  (\alpha _{1} c_{ji}p_{i}^{*}p_{i}+\alpha _{2} c_{ji}p_{i}p_{i}^{*}+\alpha _{3} p_{i}^{*}c_{ji}p_{i} +\alpha _{4} p_{i}c_{ji}p_{i}^{*}+\alpha _{5} p_{i}^{*}p_{i}c_{ji}+\alpha _{6} p_{i}p_{i}^{*}c_{ji})
\end{align*}}
that leads to identity $\Phi  ((\sum _{k=1}^{6} \alpha _{k})f_{ii}+(\alpha _{5}+\alpha _{6})f_{ij}+(\alpha _{1}+\alpha _{2})f_{ji})=\Phi  ((\alpha _{5}+\alpha _{6})b_{ij})+\Phi  ((\alpha _{1}+\alpha _{2})c_{ji}).$ Write $g_{ii}=(\sum _{k=1}^{6} \alpha _{k})f_{ii},$ $g_{ij}=(\alpha _{5}+\alpha _{6})f_{ij},$ $g_{ji}=(\alpha _{1}+\alpha _{2})f_{ji},$ $g=g_{ii}+g_{ij}+g_{ji},$ $h_{ij}=(\alpha _{5}+\alpha _{6})b_{ij}$ and $h_{ji}=(\alpha _{1}+\alpha _{2})c_{ji}.$ Then $\Phi (g)=\Phi (h_{ij}) + \Phi (h_{ji}).$ It therefore follows that, for an arbitrary $t_{ij}\in \mathcal{A}_{ij},$ 
{\allowdisplaybreaks\begin{align*}\allowdisplaybreaks
&\Phi  (\alpha _{1} g1_{\mathcal{A}}^{*}t_{ij}+\alpha _{2} gt_{ij}1_{\mathcal{A}}^{*}+\alpha _{3} 1_{\mathcal{A}}^{*}gt_{ij}+\alpha _{4} t_{ij}g1_{\mathcal{A}}^{*}+\alpha _{5} 1_{\mathcal{A}}^{*}t_{ij}g+\alpha _{6} t_{ij}1_{\mathcal{A}}^{*}g)\\
&=\Phi  (\alpha _{1} h_{ij}1_{\mathcal{A}}^{*}t_{ij}+\alpha _{2} h_{ij}t_{ij}1_{\mathcal{A}}^{*}+\alpha _{3} 1_{\mathcal{A}}^{*}h_{ij}t_{ij}+\alpha _{4} t_{ij}h_{ij}1_{\mathcal{A}}^{*}+\alpha _{5} 1_{\mathcal{A}}^{*}t_{ij}h_{ij}+\alpha _{6} t_{ij}1_{\mathcal{A}}^{*}h_{ij})\\
&+\Phi  (\alpha _{1} h_{ji}1_{\mathcal{A}}^{*}t_{ij}+\alpha _{2} h_{ji}t_{ij}1_{\mathcal{A}}^{*}+\alpha _{3} 1_{\mathcal{A}}^{*}h_{ji}t_{ij} +\alpha _{4} t_{ij}h_{ji}1_{\mathcal{A}}^{*}+\alpha _{5} 1_{\mathcal{A}}^{*}t_{ij}h_{ji}+\alpha _{6} t_{ij}1_{\mathcal{A}}^{*}h_{ji})\\
&=\Phi  ((\alpha _{1}+\alpha _{2}+\alpha _{3})h_{ji}t_{ij}+(\alpha _{4}+\alpha _{5}+\alpha _{6})t_{ij}h_{ji})
\end{align*}}
which shows that $\alpha _{1} g1_{\mathcal{A}}^{*}t_{ij}+\alpha _{2} gt_{ij}1_{\mathcal{A}}^{*}+\alpha _{3} 1_{\mathcal{A}}^{*}gt_{ij}+\alpha _{4} t_{ij}g1_{\mathcal{A}}^{*}+\alpha _{5} 1_{\mathcal{A}}^{*}t_{ij}g+\alpha _{6} t_{ij}1_{\mathcal{A}}^{*}g=(\alpha _{1}+\alpha _{2}+\alpha _{3})h_{ji}t_{ij}+(\alpha _{4}+\alpha _{5}+\alpha _{6})t_{ij}h_{ji}.$ As a result we get the identity $(\alpha _{1}+\alpha _{2}+\alpha _{3})g_{ii}t_{ij}+(\alpha _{1}+\alpha _{2}+\alpha _{3})g_{ji}t_{ij}+(\alpha _{4}+\alpha _{5}+\alpha _{6})t_{ij}g_{ji}=(\alpha _{1}+\alpha _{2}+\alpha _{3})h_{ji}t_{ij}+(\alpha _{4}+\alpha _{5}+\alpha _{6})t_{ij}h_{ji}.$ This shows that $(\alpha _{6}+\alpha _{4}+\alpha _{5})g_{ii}t_{ij}+(\alpha _{6}+\alpha _{4}+\alpha _{5})g_{ji}t_{ij}+(\alpha _{2}+\alpha _{3}+\alpha _{1})t_{ij}g_{ji}=(\alpha _{6}+\alpha _{4}+\alpha _{5})h_{ji}t_{ij}+(\alpha _{2}+\alpha _{3}+\alpha _{1})t_{ij}h_{ji},$ by Claim \ref{c21}. Adding these last two results we obtain $(\sum _{k=1}^{6} \alpha _{k})(g_{ii}t_{ij}+g_{ji}t_{ij}+t_{ij}g_{ji})=(\sum _{k=1}^{6} \alpha _{k})(h_{ji}t_{ij}+t_{ij}h_{ji})$ that allows to deduce that $g_{ii}t_{ij}=0.$ Thus $g_{ii}=0$ which implies that $f_{ii}=0.$ Using a similar reasoning as before, we prove that $f_{jj}=0.$ Next, for an arbitrary element $t_{ji}\in \mathcal{A}_{ji},$ we have 
{\allowdisplaybreaks\begin{align*}\allowdisplaybreaks
&\Phi  (\alpha _{1} f1_{\mathcal{A}}^{*}t_{ji}+\alpha _{2} ft_{ji}1_{\mathcal{A}}^{*}+\alpha _{3} 1_{\mathcal{A}}^{*}ft_{ji}+\alpha _{4} t_{ji}f1_{\mathcal{A}}^{*}+\alpha _{5} 1_{\mathcal{A}}^{*}t_{ji}f+\alpha _{6} t_{ji}1_{\mathcal{A}}^{*}f)\\
&=\Phi  (\alpha _{1} b_{ij}1_{\mathcal{A}}^{*}t_{ji}+\alpha _{2} b_{ij}t_{ji}1_{\mathcal{A}}^{*}+\alpha _{3} 1_{\mathcal{A}}^{*}b_{ij}t_{ji}+\alpha _{4} t_{ji}b_{ij}1_{\mathcal{A}}^{*}+\alpha _{5} 1_{\mathcal{A}}^{*}t_{ji}b_{ij}+\alpha _{6} t_{ji}1_{\mathcal{A}}^{*}b_{ij})\\
&+\Phi  (\alpha _{1} c_{ji}1_{\mathcal{A}}^{*}t_{ji}+\alpha _{2} c_{ji}t_{ji}1_{\mathcal{A}}^{*}+\alpha _{3} 1_{\mathcal{A}}^{*}c_{ji}t_{ji} +\alpha _{4} t_{ji}c_{ji}1_{\mathcal{A}}^{*}+\alpha _{5} 1_{\mathcal{A}}^{*}t_{ji}c_{ji}+\alpha _{6} t_{ji}1_{\mathcal{A}}^{*}c_{ji})\\
&=\Phi  ((\alpha _{1}+\alpha _{2}+\alpha _{3})b_{ij}t_{ji}+(\alpha _{4}+\alpha _{5}+\alpha _{6})t_{ji}b_{ij}).
\end{align*}}
This results that $\alpha _{1} f1_{\mathcal{A}}^{*}t_{ji}+\alpha _{2} ft_{ji}1_{\mathcal{A}}^{*}+\alpha _{3} 1_{\mathcal{A}}^{*}ft_{ji} +\alpha _{4} t_{ji}f1_{\mathcal{A}}^{*}+\alpha _{5} 1_{\mathcal{A}}^{*}t_{ji}f+\alpha _{6} t_{ji}1_{\mathcal{A}}^{*}f=(\alpha _{1}+\alpha _{2}+\alpha _{3})b_{ij}t_{ji}+(\alpha _{4}+\alpha _{5}+\alpha _{6})t_{ji}b_{ij}.$ As a consequence of this identity, we deduce that $(\alpha _{1}+\alpha _{2}+\alpha _{3})f_{ij}t_{ji}+(\alpha _{4}+\alpha _{5}+\alpha _{6})t_{ji}f_{ij}=(\alpha _{1}+\alpha _{2}+\alpha _{3})b_{ij}t_{ji}+(\alpha _{4}+\alpha _{5}+\alpha _{6})t_{ji}b_{ij}$ from which we can also deduce that $(\alpha _{6}+\alpha _{4}+\alpha _{5})f_{ij}t_{ji}+(\alpha _{2}+\alpha _{3}+\alpha _{1})t_{ji}f_{ij}=(\alpha _{6}+\alpha _{4}+\alpha _{5})b_{ij}t_{ji}+(\alpha _{2}+\alpha _{3}+\alpha _{1})t_{ji}b_{ij},$ by Claim \ref{c21}. Adding the last two identities we find $(\sum _{k=1}^{6} \alpha _{k}) (f_{ij}t_{ji}+t_{ji}f_{ij})=(\sum _{k=1}^{6} \alpha _{k})(b_{ij}t_{ji}+t_{ji}b_{ij})$ which leads to the conclusion that $f_{ij}t_{ji}=b_{ij}t_{ji}.$ Therefore, $f_{ij}=b_{ij}.$ Using a similar reasoning as before, we prove that $f_{ji}=c_{ji}.$
\end{proof}

\begin{claim}\label{c26} For arbitrary elements  $a_{ij},b_{ij}\in \mathcal{A}_{ij}$ $(i\neq j;i,j=1,2)$ we have $\Phi (a_{ij}+b_{ij})=\Phi (a_{ij})+ \Phi (b_{ij}).$
\end{claim}
\begin{proof} First, note that 
{\allowdisplaybreaks\begin{align*}\allowdisplaybreaks
&(\alpha _{1}+\alpha _{2}+\alpha _{3}) (a_{ij}+b_{ij})\\
&=\alpha _{1} (p_{i}+a_{ij})1_{\mathcal{A}}^{*}(p_{j}+b_{ij})+\alpha _{2} (p_{i}+a_{ij})(p_{j}+b_{ij})1_{\mathcal{A}}^{*}\\
&+\alpha _{3} 1_{\mathcal{A}}^{*}(p_{i}+a_{ij})(p_{j}+b_{ij})+\alpha _{4} (p_{j}+b_{ij})(p_{i}+a_{ij})1_{\mathcal{A}}^{*}\\
&+\alpha _{5} 1_{\mathcal{A}}^{*}(p_{j}+b_{ij})(p_{i}+a_{ij})+\alpha _{6} (p_{j}+b_{ij})1_{\mathcal{A}}^{*}(p_{i}+a_{ij}),
\end{align*}}
for all elements $a_{ij},b_{ij}\in \mathcal{A}_{ij}.$ Hence, by (\ref{fundident}) and Claim \ref{c24}(i), we have
{\allowdisplaybreaks\begin{align*}\allowdisplaybreaks
&\Phi ((\alpha _{1}+\alpha _{2}+\alpha _{3}) (a_{ij}+b_{ij}))\\
&=\Phi (\alpha _{1} (p_{i}+a_{ij})1_{\mathcal{A}}^{*}(p_{j}+b_{ij})+\alpha _{2} (p_{i}+a_{ij})(p_{j}+b_{ij})1_{\mathcal{A}}^{*}\\
&+\alpha _{3} 1_{\mathcal{A}}^{*}(p_{i}+a_{ij})(p_{j}+b_{ij})+\alpha _{4} (p_{j}+b_{ij})(p_{i}+a_{ij})1_{\mathcal{A}}^{*}\\
&+\alpha _{5} 1_{\mathcal{A}}^{*}(p_{j}+b_{ij})(p_{i}+a_{ij})+\alpha _{6} (p_{j}+b_{ij})1_{\mathcal{A}}^{*}(p_{i}+a_{ij}))\\
&=\alpha _{1} \Phi (p_{i}+a_{ij})\Phi (1_{\mathcal{A}})^{*}\Phi (p_{j}+b_{ij})+\alpha _{2} \Phi (p_{i}+a_{ij})\Phi (p_{j}+b_{ij})\Phi (1_{\mathcal{A}})^{*}\\
&+\alpha _{3} \Phi (1_{\mathcal{A}})^{*}\Phi (p_{i}+a_{ij})\Phi (p_{j}+b_{ij})+\alpha _{4} \Phi (p_{j}+b_{ij})\Phi (p_{i}+a_{ij})\Phi (1_{\mathcal{A}})^{*}\\
&+\alpha _{5} \Phi (1_{\mathcal{A}})^{*}\Phi (p_{j}+b_{ij})\Phi (p_{i}+a_{ij})+\alpha _{6} \Phi (p_{j}+b_{ij})\Phi (1_{\mathcal{A}})^{*}\Phi (p_{i}+a_{ij})\\
&=\alpha _{1} (\Phi (p_{i})+\Phi (a_{ij}))\Phi (1_{\mathcal{A}})^{*}(\Phi (p_{j})+\Phi (b_{ij}))\\
&+\alpha _{2} (\Phi (p_{i})+\Phi (a_{ij}))(\Phi (p_{j})+\Phi (b_{ij}))\Phi (1_{\mathcal{A}})^{*}\\
&+\alpha _{3} \Phi (1_{\mathcal{A}})^{*}(\Phi (p_{i})+\Phi (a_{ij}))(\Phi (p_{j})+\Phi (b_{ij}))\\
&+\alpha _{4} (\Phi (p_{j})+\Phi (b_{ij}))(\Phi (p_{i})+\Phi (a_{ij}))\Phi (1_{\mathcal{A}})^{*}\\
&+\alpha _{5} \Phi (1_{\mathcal{A}})^{*}(\Phi (p_{j})+\Phi (b_{ij}))(\Phi (p_{i})+\Phi (a_{ij}))\\
&+\alpha _{6} (\Phi (p_{j})+\Phi (b_{ij}))\Phi (1_{\mathcal{A}})^{*}(\Phi (p_{i})+\Phi (a_{ij}))\\
&=\alpha _{1} \Phi (p_{i})\Phi (1_{\mathcal{A}})^{*}\Phi (p_{j})+\alpha _{2} \Phi (p_{i})\Phi (p_{j})\Phi (1_{\mathcal{A}})^{*}+\alpha _{3} \Phi (1_{\mathcal{A}})^{*}\Phi (p_{i})\Phi (p_{j})\\
&+\alpha _{4} \Phi (p_{j})\Phi (p_{i})\Phi (1_{\mathcal{A}})^{*}+\alpha _{5} \Phi (1_{\mathcal{A}})^{*}\Phi (p_{j})\Phi (p_{i})+\alpha _{6} \Phi (p_{j})\Phi (1_{\mathcal{A}})^{*}\Phi (p_{i})\\
&+\alpha _{1} \Phi (a_{ij})\Phi (1_{\mathcal{A}})^{*}\Phi (p_{j})+\alpha _{2} \Phi (a_{ij})\Phi (p_{j})\Phi (1_{\mathcal{A}})^{*}+\alpha _{3} \Phi (1_{\mathcal{A}})^{*}\Phi (a_{ij})\Phi (p_{j})\\
&+\alpha _{4} \Phi (p_{j})\Phi (a_{ij})\Phi (1_{\mathcal{A}})^{*}+\alpha _{5} \Phi (1_{\mathcal{A}})^{*}\Phi (p_{j})\Phi (a_{ij})+\alpha _{6} \Phi (p_{j})\Phi (1_{\mathcal{A}})^{*}\Phi (a_{ij})\\
&+\alpha _{1} \Phi (p_{i})\Phi (1_{\mathcal{A}})^{*}\Phi (b_{ij})
+\alpha _{2} \Phi (p_{i})\Phi (b_{ij})\Phi (1_{\mathcal{A}})^{*}
+\alpha _{3} \Phi (1_{\mathcal{A}})^{*}\Phi (p_{i})\Phi (b_{ij})\\
&+\alpha _{4} \Phi (b_{ij})\Phi (p_{i})\Phi (1_{\mathcal{A}})^{*}
+\alpha _{5} \Phi (1_{\mathcal{A}})^{*}\Phi (b_{ij})\Phi (p_{i})
+\alpha _{6} \Phi (b_{ij})\Phi (1_{\mathcal{A}})^{*}\Phi (p_{i})\\
&+\alpha _{1} \Phi (a_{ij})\Phi (1_{\mathcal{A}})^{*}\Phi (b_{ij})
+\alpha _{2} \Phi (a_{ij})\Phi (b_{ij})\Phi (1_{\mathcal{A}})^{*}
+\alpha _{3} \Phi (1_{\mathcal{A}})^{*}\Phi (a_{ij})\Phi (b_{ij})\\
&+\alpha _{4} \Phi (b_{ij})\Phi (a_{ij})\Phi (1_{\mathcal{A}})^{*}
+\alpha _{5} \Phi (1_{\mathcal{A}})^{*}\Phi (b_{ij})\Phi (a_{ij})
+\alpha _{6} \Phi (b_{ij})\Phi (1_{\mathcal{A}})^{*}\Phi (a_{ij})\\
&=\Phi (\alpha _{1}p_{i}1_{\mathcal{A}}^{*}p_{j}+\alpha _{2} p_{i}p_{j}1_{\mathcal{A}}^{*}+\alpha _{3} 1_{\mathcal{A}}^{*}p_{i}p_{j}+\alpha _{4} p_{j}p_{i}1_{\mathcal{A}}^{*}+\alpha _{5} 1_{\mathcal{A}}^{*}p_{j}p_{i} \\
&+\alpha _{6} p_{j}1_{\mathcal{A}}^{*}p_{i})+\Phi (\alpha _{1} a_{ij}1_{\mathcal{A}}^{*}p_{j}+\alpha _{2} a_{ij}p_{j}1_{\mathcal{A}}^{*}+\alpha _{3} 1_{\mathcal{A}}^{*}a_{ij}p_{j}+\alpha _{4} p_{j}a_{ij}1_{\mathcal{A}}^{*}\\
&+\alpha _{5} 1_{\mathcal{A}}^{*}p_{j}a_{ij}+\alpha _{6} p_{j}1_{\mathcal{A}}^{*}a_{ij})+\Phi (\alpha _{1} p_{i}1_{\mathcal{A}}^{*}b_{ij}+\alpha _{2} p_{i}b_{ij}1_{\mathcal{A}}^{*}+\alpha _{3} 1_{\mathcal{A}}^{*}p_{i}b_{ij}\\
&+\alpha _{4} b_{ij}p_{i}1_{\mathcal{A}}^{*}
+\alpha _{5} 1_{\mathcal{A}}^{*}b_{ij}p_{i}
+\alpha _{6} b_{ij}1_{\mathcal{A}}^{*}p_{i})+\Phi (\alpha _{1} a_{ij}1_{\mathcal{A}}^{*}b_{ij}
+\alpha _{2} a_{ij}b_{ij}1_{\mathcal{A}}^{*}\\
&+\alpha _{3} 1_{\mathcal{A}}^{*}a_{ij}b_{ij}+\alpha _{4} b_{ij}a_{ij}1_{\mathcal{A}}^{*}
+\alpha _{5} 1_{\mathcal{A}}^{*}b_{ij}a_{ij}
+\alpha _{6} b_{ij}1_{\mathcal{A}}^{*}a_{ij})\\
&=\Phi ((\alpha _{1}+\alpha _{2}+\alpha _{3}) a_{ij})+\Phi ((\alpha _{1}+\alpha _{2}+\alpha _{3}) b_{ij}).
\end{align*}}
Thus
{\allowdisplaybreaks\begin{align}\allowdisplaybreaks\label{id03}
\Phi ((\alpha _{1}+\alpha _{2}+\alpha _{3}) (a_{ij}+b_{ij}))=\Phi ((\alpha _{1}+\alpha _{2}+\alpha _{3})a_{ij})+\Phi ((\alpha _{1}+\alpha _{2}+\alpha _{3}) b_{ij}).
\end{align}}
However, by Claim \ref{c21}, we have
{\allowdisplaybreaks\begin{align}\allowdisplaybreaks\label{id04}
\Phi ((\alpha _{6}+\alpha _{4}+\alpha _{5})(a_{ij}+b_{ij}))=\Phi ((\alpha _{6}+\alpha _{4}+\alpha _{5}) a_{ij})+\Phi ((\alpha _{6}+\alpha _{4}+\alpha _{5}) b_{ij}).
\end{align}}
Therefore, if $\alpha _{1}+\alpha _{2}+\alpha _{3}\neq 0,$ then the identity $\Phi (a_{ij}+b_{ij})=\Phi (a_{ij})+ \Phi (b_{ij})$ follows directly from (\ref{id03}). Otherwise, we must have $\alpha _{6}+\alpha _{4}+\alpha _{5}\neq 0$ which also leads to $\Phi (a_{ij}+b_{ij})=\Phi (a_{ij})+ \Phi (b_{ij}),$ by identity (\ref{id04}).
\end{proof}

\begin{claim}\label{c27} For arbitrary elements $a_{ii},b_{ii}\in \mathcal{A}_{ii}$ $(i=1,2),$ we have $\Phi (a_{ii}+b_{ii})=\Phi (a_{ii})+\Phi (b_{ii}).$ 
\end{claim}
\begin{proof} Take an element $f=f_{ii}+f_{ij}+f_{ji}+f_{jj}\in \mathcal{A}$ such that $\Phi (f)=\Phi (a_{ii})+\Phi (b_{ii})$. Then 
{\allowdisplaybreaks\begin{align*}\allowdisplaybreaks
&\Phi (\alpha _{1} f1_{\mathcal{A}}^{*}p_{j}+\alpha _{2} fp_{j}1_{\mathcal{A}}^{*}+\alpha _{3} 1_{\mathcal{A}}^{*}fp_{j}+\alpha _{4} p_{j}f1_{\mathcal{A}}^{*}+\alpha _{5} 1_{\mathcal{A}}^{*}p_{j}f+\alpha _{6} p_{j}1_{\mathcal{A}}^{*}f)\\
&=\Phi (\alpha _{1} a_{ii}1_{\mathcal{A}}^{*}p_{j}+\alpha _{2} a_{ii}p_{j}1_{\mathcal{A}}^{*}+\alpha _{3} 1_{\mathcal{A}}^{*}a_{ii}p_{j}+\alpha _{4} p_{j}a_{ii}1_{\mathcal{A}}^{*}+\alpha _{5} 1_{\mathcal{A}}^{*}p_{j}a_{ii}+\alpha _{6} p_{j}1_{\mathcal{A}}^{*}a_{ii})\\
&+\Phi (\alpha _{1} b_{ii}1_{\mathcal{A}}^{*}p_{j}+\alpha _{2} b_{ii}p_{j}1_{\mathcal{A}}^{*}+\alpha _{3} 1_{\mathcal{A}}^{*}b_{ii}p_{j}+\alpha _{4} p_{j}b_{ii}1_{\mathcal{A}}^{*}+\alpha _{5} 1_{\mathcal{A}}^{*}p_{j}b_{ii}+\alpha _{6} p_{j}1_{\mathcal{A}}^{*}b_{ii})\\
&=0.
\end{align*}}
This shows that $\alpha _{1} f1_{\mathcal{A}}^{*}p_{j}+\alpha _{2} fp_{j}1_{\mathcal{A}}^{*}+\alpha _{3} 1_{\mathcal{A}}^{*}fp_{j}+\alpha _{4} p_{j}f1_{\mathcal{A}}^{*}+\alpha _{5} 1_{\mathcal{A}}^{*}p_{j}f+\alpha _{6} p_{j}1_{\mathcal{A}}^{*}f=0$ which leads to the identity $(\alpha _{1}+\alpha _{2}+\alpha _{3})f_{ij}+(\alpha _{4}+\alpha _{5}+\alpha _{6})f_{ji}+(\sum _{k=1}^{6} \alpha _{k})f_{jj})=0.$ From this we deduce that $(\alpha _{6}+\alpha _{4}+\alpha _{5})f_{ij}+(\alpha _{2}+\alpha _{3}+\alpha _{1})f_{ji}+(\sum _{k=1}^{6} \alpha _{k})f_{jj})=0,$ by Claim \ref{c21}. Adding the two last equations yields $(\sum _{k=1}^{6} \alpha _{k} )(f_{ij}+f_{ji}+2f_{jj})=0$ which results in $f_{ij}=0,$ $f_{ji}=0$ and $f_{jj}=0.$ It therefore follows that $\Phi (f_{ii})=\Phi (a_{ii})+\Phi (b_{ii})$. Hence, for an arbitrary element $t_{ij}\in \mathcal{A}_{ij},$ we have
{\allowdisplaybreaks\begin{align*}\allowdisplaybreaks
&\Phi (\alpha _{1} f_{ii}1_{\mathcal{A}}^{*}t_{ij}+\alpha _{2} f_{ii}t_{ij}1_{\mathcal{A}}^{*}+\alpha _{3} 1_{\mathcal{A}}^{*}f_{ii}t_{ij} +\alpha _{4} t_{ij}f_{ii}1_{\mathcal{A}}^{*}+\alpha _{5} 1_{\mathcal{A}}^{*}t_{ij}f_{ii}+\alpha _{6} t_{ij}1_{\mathcal{A}}^{*}f_{ii})\\
&=\Phi (\alpha _{1} a_{ii}1_{\mathcal{A}}^{*}t_{ij}+\alpha _{2} a_{ii}t_{ij}1_{\mathcal{A}}^{*}+\alpha _{3} 1_{\mathcal{A}}^{*}a_{ii}t_{ij}+\alpha _{4} t_{ij}a_{ii}1_{\mathcal{A}}^{*}+\alpha _{5} 1_{\mathcal{A}}^{*}t_{ij}a_{ii}+\alpha _{6} t_{ij}1_{\mathcal{A}}^{*}a_{ii})\\ 
&+\Phi (\alpha _{1} b_{ii}1_{\mathcal{A}}^{*}t_{ij}+\alpha _{2} b_{ii}t_{ij}1_{\mathcal{A}}^{*}+\alpha _{3} 1_{\mathcal{A}}^{*}b_{ii}t_{ij}+\alpha _{4} t_{ij}b_{ii}1_{\mathcal{A}}^{*}+\alpha _{5} 1_{\mathcal{A}}^{*}t_{ij}b_{ii}+\alpha _{6} t_{ij}1_{\mathcal{A}}^{*}b_{ii})\\
&=\Phi ((\alpha _{1} +\alpha _{2}+\alpha _{3})a_{ii}t_{ij})+\Phi ((\alpha _{1} +\alpha _{2}+\alpha _{3})b_{ii}t_{ij})\\
&=\Phi ((\alpha _{1} +\alpha _{2}+\alpha _{3})(a_{ii}+b_{ii})t_{ij}),
\end{align*}}
by Claim \ref{c26}, which results that $(\alpha _{1} +\alpha _{2}+\alpha _{3})f_{ii}t_{ij}=(\alpha _{1} +\alpha _{2}+\alpha _{3})(a_{ii}+b_{ii})t_{ij}.$ This makes it possible to deduce that $(\alpha _{6} +\alpha _{4}+\alpha _{5})f_{ii}t_{ij}=(\alpha _{6} +\alpha _{4}+\alpha _{5})(a_{ii}+b_{ii})t_{ij},$ by Claim \ref{c21}. Thus, adding the two last identities we get $(\sum _{k=1}^{6} \alpha _{k}) f_{ii}t_{ij}=(\sum _{k=1}^{6} \alpha _{k} )(a_{ii}+b_{ii})t_{ij}$ which yields $f_{ii}t_{ij}=(a_{ii}+b_{ii})t_{ij}.$ As consequence, we obtain $f_{ii}=a_{ii}+b_{ii}.$
\end{proof}

\begin{claim}\label{c28} For arbitrary elements $a_{ii}\in \mathcal{A}_{ii}$, $b_{ij}\in \mathcal{A}_{ij}$, $c_{ji}\in \mathcal{A}_{ji}$ and  $d_{jj}\in \mathcal{A}_{jj}$ $(i\neq j;i,j=1,2)$ the following holds: (i) $\Phi (a_{ii}+b_{ij}+c_{ji})=\Phi (a_{ii})+\Phi (b_{ij})+\Phi (c_{ji})$ and (ii) $\Phi (b_{ij}+c_{ji}+d_{jj})=\Phi (b_{ij})+\Phi (c_{ji})+\Phi (d_{jj}).$
\end{claim}
\begin{proof} Take an element $f=f_{ii}+f_{ij}+f_{ji}+f_{jj}\in \mathcal{A}$ such that $\Phi (f)=\Phi (a_{ii})+\Phi (b_{ij})+\Phi (c_{ji})$ and write $\Phi (f)=\Phi (a_{ii})+\Phi (b_{ij}+c_{ji}),$ by Claim \ref{c25}. Then
{\allowdisplaybreaks\begin{align*}\allowdisplaybreaks
&\Phi (\alpha _{1} f1_{\mathcal{A}}^{*}p_{j}+\alpha _{2} fp_{j}1_{\mathcal{A}}^{*}+\alpha _{3} 1_{\mathcal{A}}^{*}fp_{j} +\alpha _{4} p_{j}f1_{\mathcal{A}}^{*}+\alpha _{5} 1_{\mathcal{A}}^{*}p_{j}f+\alpha _{6} p_{j}1_{\mathcal{A}}^{*}f)\\
&=\Phi (\alpha _{1} a_{ii}1_{\mathcal{A}}^{*}p_{j}+\alpha _{2} a_{ii}p_{j}1_{\mathcal{A}}^{*}+\alpha _{3} 1_{\mathcal{A}}^{*}a_{ii}p_{j}+\alpha _{4} p_{j}a_{ii}1_{\mathcal{A}}^{*}+\alpha _{5} 1_{\mathcal{A}}^{*}p_{j}a_{ii}+\alpha _{6} p_{j}1_{\mathcal{A}}^{*}a_{ii})\\
&+\Phi (\alpha _{1} (b_{ij}+c_{ji})1_{\mathcal{A}}^{*}p_{j}+\alpha _{2} (b_{ij}+c_{ji})p_{j}1_{\mathcal{A}}^{*}+\alpha _{3} 1_{\mathcal{A}}^{*}(b_{ij}+c_{ji})p_{j}\\
&+\alpha _{4} p_{j}(b_{ij}+c_{ji})1_{\mathcal{A}}^{*}+\alpha _{5} 1_{\mathcal{A}}^{*}p_{j}(b_{ij}+c_{ji})+\alpha _{6} p_{j}1_{\mathcal{A}}^{*}(b_{ij}+c_{ji}))\\
&=\Phi ((\alpha _{1}+\alpha _{2}+\alpha _{3})b_{ij}+(\alpha _{4}+\alpha _{5}+\alpha _{6})c_{ji}).
\end{align*}}
It follows directly from this that $\alpha _{1} f1_{\mathcal{A}}^{*}p_{j}+\alpha _{2} fp_{j}1_{\mathcal{A}}^{*}+\alpha _{3} 1_{\mathcal{A}}^{*}fp_{j} +\alpha _{4} p_{j}f1_{\mathcal{A}}^{*}+\alpha _{5} 1_{\mathcal{A}}^{*}p_{j}f+\alpha _{6} p_{j}1_{\mathcal{A}}^{*}f=(\alpha _{1}+\alpha _{2}+\alpha _{3})b_{ij}+(\alpha _{4}+\alpha _{5}+\alpha _{6})c_{ji}$ which results in $(\alpha _{1}+\alpha _{2}+\alpha _{3})f_{ij}+(\alpha _{4}+\alpha _{5}+\alpha _{6})f_{ji}+ (\sum _{k=1}^{6} \alpha _{k})f_{jj}=(\alpha _{1}+\alpha _{2}+\alpha _{3})b_{ij}+(\alpha _{4}+\alpha _{5}+\alpha _{6})c_{ji}.$ As a result, we can apply Claim \ref{c21} to conclude that $(\alpha _{6}+\alpha _{4}+\alpha _{5})f_{ij}+(\alpha _{2}+\alpha _{3}+\alpha _{1})f_{ji}+ (\sum _{k=1}^{6} \alpha _{k})f_{jj}=(\alpha _{6}+\alpha _{4}+\alpha _{5})b_{ij}+(\alpha _{2}+\alpha _{3}+\alpha _{1})c_{ji}.$ Adding these two identities we obtain $(\sum _{k=1}^{6} \alpha _{k})(f_{ij}+f_{ji}+2f_{jj})=(\sum _{k=1}^{6} \alpha _{k})(b_{ij}+c_{ji})$ which shows that $f_{ij}=b_{ij},$ $f_{ji}=c_{ji}$ and $f_{jj}=0.$ Next, write $\Phi (f)=\Phi (a_{ii}+b_{ij})+\Phi (c_{ji}),$ by Claim \ref{c24}(i). For an arbitrary element $t_{ji}\in \mathcal{A}_{ji},$ we have 
{\allowdisplaybreaks\begin{align*}\allowdisplaybreaks
&\Phi (\alpha _{1} f1_{\mathcal{A}}^{*}t_{ji}+\alpha _{2} ft_{ji}1_{\mathcal{A}}^{*}+\alpha _{3} 1_{\mathcal{A}}^{*}ft_{ji} +\alpha _{4} t_{ji}f1_{\mathcal{A}}^{*}+\alpha _{5} 1_{\mathcal{A}}^{*}t_{ji}f+\alpha _{6} t_{ji}1_{\mathcal{A}}^{*}f)\\
&=\Phi (\alpha _{1} (a_{ii}+b_{ij})1_{\mathcal{A}}^{*}t_{ji}+\alpha _{2} (a_{ii}+b_{ij})t_{ji}1_{\mathcal{A}}^{*}+\alpha _{3} 1_{\mathcal{A}}^{*}(a_{ii}+b_{ij})t_{ji}\\
&+\alpha _{4} t_{ji}(a_{ii}+b_{ij})1_{\mathcal{A}}^{*}+\alpha _{5} 1_{\mathcal{A}}^{*}t_{ji}(a_{ii}+b_{ij})+\alpha _{6} t_{ji}1_{\mathcal{A}}^{*}(a_{ii}+b_{ij}))\\
&+\Phi (\alpha _{1} c_{ji}1_{\mathcal{A}}^{*}t_{ji}+\alpha _{2} c_{ji}t_{ji}1_{\mathcal{A}}^{*}+\alpha _{3} 1_{\mathcal{A}}^{*}c_{ji}t_{ji}+\alpha _{4} t_{ji}c_{ji}1_{\mathcal{A}}^{*}+\alpha _{5} 1_{\mathcal{A}}^{*}t_{ji}c_{ji}+\alpha _{6} t_{ji}1_{\mathcal{A}}^{*}c_{ji})\\
&=\Phi ((\alpha _{1}+\alpha _{2}+\alpha _{3})b_{ij}t_{ji}+(\alpha _{4}+\alpha _{5}+\alpha _{6})t_{ji}a_{ii}+(\alpha _{4}+\alpha _{5}+\alpha _{6})t_{ji}b_{ij}).
\end{align*}}
This implies that $\alpha _{1} f1_{\mathcal{A}}^{*}t_{ji}+\alpha _{2} ft_{ji}1_{\mathcal{A}}^{*}+\alpha _{3} 1_{\mathcal{A}}^{*}ft_{ji} +\alpha _{4} t_{ji}f1_{\mathcal{A}}^{*}+\alpha _{5} 1_{\mathcal{A}}^{*}t_{ji}f+\alpha _{6} t_{ji}1_{\mathcal{A}}^{*}f=(\alpha _{1}+\alpha _{2}+\alpha _{3})b_{ij}t_{ji}+(\alpha _{4}+\alpha _{5}+\alpha _{6})t_{ji}a_{ii}+(\alpha _{4}+\alpha _{5}+\alpha _{6})t_{ji}b_{ij}$ which results in $(\alpha _{4}+\alpha _{5}+\alpha _{6})t_{ji}f_{ii}=(\alpha _{4}+\alpha _{5}+\alpha _{6})t_{ji}a_{ii}.$ Now we can apply Claim \ref{c21} to conclude that $(\alpha _{2}+\alpha _{3}+\alpha _{1})t_{ji}f_{ii}=(\alpha _{2}+\alpha _{3}+\alpha _{1})t_{ji}a_{ii}.$ Thus, adding these two last identities we get $(\sum _{k=1}^{6} \alpha _{k})t_{ji}a_{ii}=(\sum _{k=1}^{6} \alpha _{k})t_{ji}a_{ii}$ which leads to $t_{ji}f_{ii}=t_{ji}a_{ii}.$ Therefore, $f_{ii}=a_{ii}.$

By an entirely similar reasoning, we prove the case (ii).
\end{proof}

\begin{claim}\label{c29}  For arbitrary elements $a_{ii}\in \mathcal{A}_{ii}$, $b_{ij}\in \mathcal{A}_{ij}$, $c_{ji}\in \mathcal{A}_{ji}$ and  $d_{jj}\in \mathcal{A}_{jj}$ $(i\neq j;i,j=1,2)$ the following holds $\Phi (a_{ii}+b_{ij}+c_{ji}+d_{jj})=\Phi (a_{ii})+\Phi (b_{ij})+\Phi (c_{ji})+\Phi (d_{jj}).$
\end{claim}
\begin{proof} Consider an element $f=f_{ii}+f_{ij}+f_{ji}+f_{jj}\in \mathcal{A}$ such that $\Phi (f)=\Phi (a_{ii})+\Phi (b_{ij})+\Phi (c_{ji})+\Phi (d_{jj})$ and write $\Phi (f)=\Phi (a_{ii}+b_{ij}+c_{ji})+\Phi (d_{jj})$, by Claim \ref{c28}(i). By Claim \ref{c25} we have
{\allowdisplaybreaks\begin{align*}\allowdisplaybreaks
&\Phi (\alpha _{1} f1_{\mathcal{A}}^{*}p_{i}+\alpha _{2} fp_{i}1_{\mathcal{A}}^{*}+\alpha _{3} 1_{\mathcal{A}}^{*}fp_{i} +\alpha _{4} p_{i}f1_{\mathcal{A}}^{*}+\alpha _{5} 1_{\mathcal{A}}^{*}p_{i}f+\alpha _{6} p_{i}1_{\mathcal{A}}^{*}f)\\
&=\Phi (\alpha _{1} (a_{ii}+b_{ij}+c_{ji})1_{\mathcal{A}}^{*}p_{i}+\alpha _{2} (a_{ii}+b_{ij}+c_{ji})p_{i}1_{\mathcal{A}}^{*}+\alpha _{3} 1_{\mathcal{A}}^{*}(a_{ii}+b_{ij}+c_{ji})p_{i}\\
&+\alpha _{4} p_{i}(a_{ii}+b_{ij}+c_{ji})1_{\mathcal{A}}^{*}+\alpha _{5} 1_{\mathcal{A}}^{*}p_{i}(a_{ii}+b_{ij}+c_{ji})+\alpha _{6} p_{i}1_{\mathcal{A}}^{*}(a_{ii}+b_{ij}+c_{ji}))\\
&+\Phi (\alpha _{1} d_{jj}1_{\mathcal{A}}^{*}p_{i}+\alpha _{2} d_{jj}p_{i}1_{\mathcal{A}}^{*}+\alpha _{3} 1_{\mathcal{A}}^{*}d_{jj}p_{i}+\alpha _{4} p_{i}d_{jj}1_{\mathcal{A}}^{*}+\alpha _{5} 1_{\mathcal{A}}^{*}p_{i}d_{jj}+\alpha _{6} p_{i}1_{\mathcal{A}}^{*}d_{jj})\\
&=\Phi ((\textstyle \sum _{k=1}^{6} \alpha _{k})a_{ii}+(\alpha _{4}+\alpha _{5}+\alpha _{6})b_{ij}+(\alpha _{1}+\alpha _{2}+\alpha _{3})c_{ji}).
\end{align*}}
It follows that $\alpha _{1} f1_{\mathcal{A}}^{*}p_{i}+\alpha _{2} fp_{i}1_{\mathcal{A}}^{*}+\alpha _{3} 1_{\mathcal{A}}^{*}fp_{i} +\alpha _{4} p_{i}f1_{\mathcal{A}}^{*}+\alpha _{5} 1_{\mathcal{A}}^{*}p_{i}f+\alpha _{6} p_{i}1_{\mathcal{A}}^{*}f=(\sum _{k=1}^{6} \alpha _{k})a_{ii}+(\alpha _{4}+\alpha _{5}+\alpha _{6})b_{ij}+(\alpha _{1}+\alpha _{2}+\alpha _{3})c_{ji}$ which implies that $(\sum _{k=1}^{6} \alpha _{k})f_{ii}+(\alpha _{4}+\alpha _{5}+\alpha _{6})f_{ij}+(\alpha _{1}+\alpha _{2}+\alpha _{3})f_{ji}=(\sum _{k=1}^{6} \alpha _{k})a_{ii}+(\alpha _{4}+\alpha _{5}+\alpha _{6})b_{ij}+(\alpha _{1}+\alpha _{2}+\alpha _{3})c_{ji}$ and in view of Claim \ref{c21} we arrive at $(\sum _{k=1}^{6} \alpha _{k})f_{ii}+(\alpha _{2}+\alpha _{3}+\alpha _{1})f_{ij}+(\alpha _{6}+\alpha _{4}+\alpha _{5})f_{ji}=(\sum _{k=1}^{6} \alpha _{k})a_{ii}+(\alpha _{2}+\alpha _{3}+\alpha _{1})b_{ij}+(\alpha _{6}+\alpha _{4}+\alpha _{5})c_{ji}.$ Adding these two identities we have $(\sum _{k=1}^{6} \alpha _{k})(2f_{ii}+f_{ij}+f_{ji})=(\sum _{k=1}^{6} \alpha _{k})(2a_{ii}+b_{ij}+c_{ji})$ which shows that $f_{ii}=a_{ii},$ $f_{ij}=b_{ij}$ and $f_{ji}=c_{ji}.$  Using a similar reasoning as before, we prove that $f_{jj}=d_{jj}.$
\end{proof}

\begin{claim}\label{c210} $\Phi $ is an additive mapping.
\end{claim}
\begin{proof} The result is a direct consequence of Claims \ref{c26}, \ref{c27} and \ref{c29}.
\end{proof}

To prove the second part of the Theorem \ref{thm21}, we assume that the element $\Phi (1_{\mathcal{A}})$ is a projection of $\mathcal{B}.$

\begin{claim}\label{c211} (i) $\Phi (1_{\mathcal{A}})=1_{\mathcal{B}},$ (ii) $\Phi ((\sum _{k=1}^{6} \alpha _{k})a)=(\sum _{k=1}^{6} \alpha _{k})\Phi (a),$ for all element $a\in \mathcal{A},$ and (iii) $\Phi (b^{*})=\Phi (b)^{*},$ for all element $b\in \mathcal{A}.$  
\end{claim}
\begin{proof} First, we observe that 
{\allowdisplaybreaks\begin{align*}\allowdisplaybreaks
&\Phi ((\textstyle \sum _{k=1}^{6} \alpha _{k})1_{\mathcal{A}})=\Phi (\alpha _{1} 1_{\mathcal{A}}1_{\mathcal{A}}^{*}1_{\mathcal{A}}+\alpha _{2} 1_{\mathcal{A}}1_{\mathcal{A}}1_{\mathcal{A}}^{*}+\alpha _{3} 1_{\mathcal{A}}^{*}1_{\mathcal{A}}1_{\mathcal{A}} +\alpha _{4} 1_{\mathcal{A}}1_{\mathcal{A}}1_{\mathcal{A}}^{*}\\
&+\alpha _{5} 1_{\mathcal{A}}^{*}1_{\mathcal{A}}1_{\mathcal{A}}+\alpha _{6} 1_{\mathcal{A}}1_{\mathcal{A}}^{*}1_{\mathcal{A}})=\alpha _{1} \Phi (1_{\mathcal{A}})\Phi (1_{\mathcal{A}})^{*}\Phi (1_{\mathcal{A}})+\alpha _{2} \Phi (1_{\mathcal{A}})\Phi (1_{\mathcal{A}})\Phi (1_{\mathcal{A}})^{*}\\
&+\alpha _{3} \Phi (1_{\mathcal{A}})^{*}\Phi (1_{\mathcal{A}})\Phi (1_{\mathcal{A}})+\alpha _{4} \Phi (1_{\mathcal{A}})\Phi (1_{\mathcal{A}})\Phi (1_{\mathcal{A}})^{*}+\alpha _{5} \Phi (1_{\mathcal{A}})^{*}\Phi (1_{\mathcal{A}})\Phi (1_{\mathcal{A}})\\
&+\alpha _{6}\Phi (1_{\mathcal{A}})\Phi (1_{\mathcal{A}})^{*}\Phi (1_{\mathcal{A}})=(\textstyle \sum _{k=1}^{6} \alpha _{k})\Phi (1_{\mathcal{A}}).
\end{align*}}
Thus, if $b\in \mathcal{A}$ is an element such that $\Phi (b)=1_{\mathcal{B}},$ then 
{\allowdisplaybreaks\begin{align*}\allowdisplaybreaks
&\Phi ((\textstyle \sum _{k=1}^{6} \alpha _{k})b^{*})=\Phi (\alpha _{1} 1_{\mathcal{A}}b^{*}1_{\mathcal{A}}+\alpha _{2} 1_{\mathcal{A}}1_{\mathcal{A}}b^{*}+\alpha _{3} b^{*}1_{\mathcal{A}}1_{\mathcal{A}} +\alpha _{4} 1_{\mathcal{A}}1_{\mathcal{A}}b^{*}\\
&+\alpha _{5} b^{*}1_{\mathcal{A}}1_{\mathcal{A}}+\alpha _{6} 1_{\mathcal{A}}b^{*}1_{\mathcal{A}})=\alpha _{1} \Phi (1_{\mathcal{A}})\Phi (b)^{*}\Phi (1_{\mathcal{A}})+\alpha _{2} \Phi (1_{\mathcal{A}})\Phi (1_{\mathcal{A}})\Phi (b)^{*}\\
&+\alpha _{3} \Phi (b)^{*}\Phi (1_{\mathcal{A}})\Phi (1_{\mathcal{A}})+\alpha _{4} \Phi (1_{\mathcal{A}})\Phi (1_{\mathcal{A}})\Phi (b)^{*}+\alpha _{5} \Phi (b)^{*}\Phi (1_{\mathcal{A}})\Phi (1_{\mathcal{A}})\\
&+\alpha _{6}\Phi (1_{\mathcal{A}})\Phi (b)^{*}\Phi (1_{\mathcal{A}})=(\textstyle \sum _{k=1}^{6} \alpha _{k}) \Phi (1_{\mathcal{A}})=\Phi ((\textstyle \sum _{k=1}^{6} \alpha _{k})1_{\mathcal{A}}).
\end{align*}}
This shows that $b^{*}=1_{\mathcal{A}}$ which leads to $b=1_{\mathcal{A}}.$ For this reason, for an arbitrary element $a\in \mathcal{A},$ replace $b$ and $c$ by $1_{\mathcal{A}},$ respectively, in (\ref{fundident}). Then  
{\allowdisplaybreaks\begin{align*}\allowdisplaybreaks
&\Phi ((\textstyle \sum _{k=1}^{6} \alpha _{k})a)=\Phi (\alpha _{1} a1_{\mathcal{A}}^{*}1_{\mathcal{A}}+\alpha _{2} a1_{\mathcal{A}}1_{\mathcal{A}}^{*}+\alpha _{3} 1_{\mathcal{A}}^{*}a1_{\mathcal{A}} +\alpha _{4} 1_{\mathcal{A}}a1_{\mathcal{A}}^{*}+\alpha _{5} 1_{\mathcal{A}}^{*}1_{\mathcal{A}}a\\
&+\alpha _{6} 1_{\mathcal{A}}1_{\mathcal{A}}^{*}a)=\alpha _{1}\Phi (a)\Phi (1_{\mathcal{A}})^{*}\Phi (1_{\mathcal{A}})+\alpha _{2}\Phi (a)\Phi (1_{\mathcal{A}})\Phi (1_{\mathcal{A}})^{*}\\
&+\alpha _{3}\Phi (1_{\mathcal{A}})^{*}\Phi (a)\Phi (1_{\mathcal{A}})+\alpha _{4}\Phi (1_{\mathcal{A}})\Phi (a)\Phi (1_{\mathcal{A}})^{*}+\alpha _{5}\Phi (1_{\mathcal{A}})^{*}\Phi (1_{\mathcal{A}})\Phi (a)\\
&+\alpha _{6}\Phi (1_{\mathcal{A}})\Phi (1_{\mathcal{A}})^{*}\Phi (a)=(\textstyle \sum _{k=1}^{6} \alpha _{k})\Phi (a).
\end{align*}}
Moreover, for an arbitrary element $b\in \mathcal{A},$ replace $a$ and $c$ by $1_{\mathcal{A}},$ respectively, in (\ref{fundident}). Then we arrive at $\Phi (b^{*})=\Phi (b)^{*}.$ 
\end{proof}

\begin{claim}\label{c212} $\Phi $ is a $\ast $-Jordan multiplicative mapping.
\end{claim}
\begin{proof} For arbitrary elements $a,c\in \mathcal{A}$ we have
{\allowdisplaybreaks\begin{align*}\allowdisplaybreaks
&\Phi (\alpha _{1} a1_{\mathcal{A}}^{*}c+\alpha _{2} ac1_{\mathcal{A}}^{*}+\alpha _{3} 1_{\mathcal{A}}^{*}ac +\alpha _{4} ca1_{\mathcal{A}}^{*}+\alpha _{5} 1_{\mathcal{A}}^{*}ca+\alpha _{6} c1_{\mathcal{A}}^{*}a)\\
&=\alpha _{1}\Phi (a)\Phi (1_{\mathcal{A}})^{*}\Phi (c)+\alpha _{2}\Phi (a)\Phi (c)\Phi (1_{\mathcal{A}})^{*}+\alpha _{3}\Phi (1_{\mathcal{A}})^{*}\Phi (a)\Phi (c)\\
&+\alpha _{4}\Phi (c)\Phi (a)\Phi (1_{\mathcal{A}})^{*}+\alpha _{5}\Phi (1_{\mathcal{A}})^{*}\Phi (c)\Phi (a)+\alpha _{6}\Phi (c)\Phi (1_{\mathcal{A}})^{*}\Phi (a)
\end{align*}}
which results that
{\allowdisplaybreaks\begin{align}\allowdisplaybreaks\label{ident05}
&\Phi ((\alpha _{1}+\alpha _{2}+\alpha _{3}) ac +(\alpha _{4} +\alpha _{5}+\alpha _{6}) ca)\nonumber\\
&=(\alpha _{1}+\alpha _{2}+\alpha _{3})\Phi (a)\Phi (c)+(\alpha _{4}+\alpha _{5}+\alpha _{6})\Phi (c)\Phi (a).
\end{align}}
Similarly, by Claim \ref{c21} we obtain
{\allowdisplaybreaks\begin{align}\allowdisplaybreaks\label{ident06}
&\Phi ((\alpha _{6}+\alpha _{4}+\alpha _{5}) ac +(\alpha _{2} +\alpha _{3}+\alpha _{1}) ca)\nonumber\\
&=(\alpha _{6}+\alpha _{4}+\alpha _{5})\Phi (a)\Phi (c)+(\alpha _{2}+\alpha _{3}+\alpha _{1})\Phi (c)\Phi (a),
\end{align}}
Adding the identities (\ref{ident05}) and (\ref{ident06}) and using the Claims \ref{c210} and \ref{c211}(ii) we arrive at
{\allowdisplaybreaks\begin{align}\allowdisplaybreaks\label{ident07}
&\Phi (ac+ca)=\Phi (a)\Phi (c)+\Phi (c)\Phi (a).
\end{align}}
Thus the result follows in view of Claim \ref{c211}(iii).
\end{proof}

Now, we assume that $\mathcal{B}$ is prime and $\phi (1_{\mathcal{A}})$ is a projection of $\mathcal{B}.$ From this we can easily deduce the following result.

\begin{claim}\label{} $\Phi $ is either a $\ast $-ring isomorphism or an $\ast $-ring anti-isomorphism.
\end{claim}
\begin{proof} The result is a direct consequence of the Claim \ref{c212} and \cite[Theorem H]{Herstein}. 
\end{proof}

The second main result of this paper reads as follows.

\begin{theorem}\label{thm22}  Let $\{\alpha _{k}\}_{k=1}^{6}$ be complex numbers satisfying the condition $\sum _{k=1}^{6} \alpha _{k} \neq 0,$ $\mathcal{A}$ and $\mathcal{B}$ two unital complex $\ast $-algebras with $1_{\mathcal{A}}$ and $1_{\mathcal{B}}$ their multiplicative identities, respectively, and such that $\mathcal{A}$ is prime and has a nontrivial projection. Let $\Phi :\mathcal{A}\rightarrow \mathcal{B}$ be a bijective mapping preserving sum of triple products $\alpha _{1} ab^{*}c+\alpha _{2} acb^{*}+\alpha _{3} b^{*}ac +\alpha _{4} cab^{*}+\alpha _{5} b^{*}ca+\alpha _{6} cb^{*}a$ such that $\Phi (1_{\mathcal{A}})$ is a projection and satisfying at least one of the following conditions: (i) $\sum _{k=1,2,3} \alpha _{k}-\sum _{k=4,5,6} \alpha _{k} \neq 0$ and $\Phi ((\sum _{k=1,2,3} \alpha _{k}-\sum _{k=4,5,6} \alpha _{k})a)=(\sum _{k=1,2,3} \alpha _{k}-\sum _{k=4,5,6} \alpha _{k})\Phi (a),$ for all element $a\in \mathcal{A},$ (ii) $\sum _{k=1,3,5} \alpha _{k}-\sum _{k=2,4,6} \alpha _{k} \neq 0$ and $\Phi ((\sum _{k=1,3,5} \alpha _{k}-\sum _{k=2,4,6} \alpha _{k})a)=(\sum _{k=1,3,5} \alpha _{k}-\sum _{k=2,4,6} \alpha _{k})\Phi (a),$ for all element $a\in \mathcal{A},$ (iii) $\sum _{k=1,2,4} \alpha _{k}-\sum _{k=3,5,6} \alpha _{k} \neq 0$ and $\Phi ((\sum _{k=1,2,4} \alpha _{k}-\sum _{k=3,5,6} \alpha _{k})a)=(\sum _{k=1,2,4} \alpha _{k}-\sum _{k=3,5,6} \alpha _{k})\Phi (a),$ for all element $a\in \mathcal{A}.$

Then $\Phi $ is a $\ast $-ring isomorphism.
\end{theorem}
\begin{proof} To prove the Theorem it is enough to show that $\Phi $ is a multiplicative mapping, in view of the Claims \ref{c210} and \ref{c211}(iii). First case: $\sum _{k=1,2,3} \alpha _{k}-\sum _{k=4,5,6} \alpha _{k} \neq 0.$ Subtracting (\ref{ident06}) from (\ref{ident05}), we have
{\allowdisplaybreaks\begin{align*}\allowdisplaybreaks
&\Phi ((\alpha _{1}+\alpha _{2}+\alpha _{3}-\alpha _{4}-\alpha _{5}-\alpha _{6}) ac-(\alpha _{1}+\alpha _{2}+\alpha _{3}-\alpha _{4}-\alpha _{5}-\alpha _{6}) ca)\\
&=(\alpha _{1}+\alpha _{2}+\alpha _{3}-\alpha _{4}-\alpha _{5}-\alpha _{6})\Phi (a)\Phi (c)\\
&\hspace{4.0cm}-(\alpha _{1}+\alpha _{2}+\alpha _{3}-\alpha _{4}-\alpha _{5}-\alpha _{6})\Phi (c)\Phi (a)
\end{align*}}
which leads to 
{\allowdisplaybreaks\begin{align}\allowdisplaybreaks\label{ident08}
&\Phi (ac-ca)=\Phi (a)\Phi (c)-\Phi (c)\Phi (a),
\end{align}}
in view of hypothesis. From (\ref{ident07}) and (\ref{ident08}), we arrive at $\Phi (ac)=\Phi (a)\Phi (c),$ for all elements $a,c\in \mathcal{A}.$ Second case: $\sum _{k=1,3,5} \alpha _{k}-\sum _{k=2,4,6} \alpha _{k} \neq 0.$ Replacing $a$ by $1_{\mathcal{A}}$ and $b$ by $b^{*},$ in (\ref{fundident}), we get
{\allowdisplaybreaks\begin{align*}\allowdisplaybreaks
&\Phi (\alpha _{1} 1_{\mathcal{A}}bc+\alpha _{2} 1_{\mathcal{A}}cb+\alpha _{3} b1_{\mathcal{A}}c +\alpha _{4} c1_{\mathcal{A}}b+\alpha _{5} bc1_{\mathcal{A}}+\alpha _{6} cb1_{\mathcal{A}})\nonumber \\
&=\alpha _{1}\Phi (1_{\mathcal{A}})\Phi (b)\Phi (c)+\alpha _{2}\Phi (1_{\mathcal{A}})\Phi (c)\Phi (b)+\alpha _{3}\Phi (b)\Phi (1_{\mathcal{A}})\Phi (c)\nonumber \\
&+\alpha _{4}\Phi (c)\Phi (1_{\mathcal{A}})\Phi (b)+\alpha _{5}\Phi (b)\Phi (c)\Phi (1_{\mathcal{A}})+\alpha _{6}\Phi (c)\Phi (b)\Phi (1_{\mathcal{A}}),
\end{align*}}
which leads to the identity
{\allowdisplaybreaks\begin{align*}\allowdisplaybreaks
&\Phi ((\alpha _{1}+\alpha _{3}+\alpha _{5}) bc +(\alpha _{2} +\alpha _{4}+\alpha _{6}) cb)\\
&=(\alpha _{1}+\alpha _{3}+\alpha _{5})\Phi (b)\Phi (c)+(\alpha _{2}+\alpha _{4}+\alpha _{6})\Phi (c)\Phi (b).
\end{align*}}
Now, replacing $b$ by $c$ and $c$ by $b,$ in the last identity, we get 
{\allowdisplaybreaks\begin{align*}\allowdisplaybreaks
&\Phi ((\alpha _{2} +\alpha _{4}+\alpha _{6}) bc+(\alpha _{1}+\alpha _{3}+\alpha _{5}) cb)\\
&=(\alpha _{2}+\alpha _{4}+\alpha _{6})\Phi (b)\Phi (c)+(\alpha _{1}+\alpha _{3}+\alpha _{5})\Phi (c)\Phi (b).
\end{align*}}
subtracting the last identity from the previous one, we have
{\allowdisplaybreaks\begin{align*}\allowdisplaybreaks
&\Phi ((\alpha _{1}+\alpha _{3}+\alpha _{5}-\alpha _{2}-\alpha _{4}-\alpha _{6}) bc-(\alpha _{1}+\alpha _{3}+\alpha _{5}-\alpha _{2}-\alpha _{4}-\alpha _{6}) cb)\\
&=(\alpha _{1}+\alpha _{3}+\alpha _{5}-\alpha _{2}-\alpha _{4}-\alpha _{6})\Phi (b)\Phi (c)\\
&\hspace{4.0cm} +(\alpha _{1}+\alpha _{3}+\alpha _{5}-\alpha _{2}-\alpha _{4}-\alpha _{6})\Phi (c)\Phi (b)
\end{align*}}
which implies that
{\allowdisplaybreaks\begin{align}\allowdisplaybreaks\label{ident09}
&\Phi (bc-cb)=\Phi (b)\Phi (c)-\Phi (c)\Phi (b),
\end{align}}
for all elements $b,c\in \mathcal{A}.$ The identities (\ref{ident07}) and (\ref{ident09}) show that $\Phi (bc)=\Phi (b)\Phi (c),$ for all elements $b,c\in \mathcal{A}.$ Third case: $\sum _{k=1,2,4} \alpha _{k}-\sum _{k=3,5,6} \alpha _{k} \neq 0.$ Using reasoning similar to the second case, we can conclude that $\Phi (ab)=\Phi (a)\Phi (b),$ for all elements $a,b\in \mathcal{A}.$

From what we just saw above we deduce that $\Phi $ is a multiplicative mapping.
\end{proof}

From Theorems \ref{thm21} and \ref{thm22} we can deduce the following results.

\begin{corollary} Let $\mathcal{A}$ and $\mathcal{B}$ be two unital complex $\ast $-algebras with $1_{\mathcal{A}}$ and $1_{\mathcal{B}}$ their multiplicative identities, respectively, and such that $\mathcal{A}$ is prime and has a nontrivial projection. Then every bijective mapping $\Phi :\mathcal{A}\rightarrow \mathcal{B}$ preserving triple product $a\filledsquare _{\eta }b\filledsquare _{\nu }c$ (resp., preserving mixed product $a\filledsquare _{\eta }b\circ _{\nu }c$), where $\eta $ and $\nu $ are nonzero complex numbers satifying the conditions $\overline{\eta }\neq -1$ and $\nu \neq -1$ (resp., $\eta \neq -1$ and $\nu \neq -1$), is additive. Moreover, (i) if $\Phi (1_{\mathcal{A}})$ is a projection, then $\Phi $ is a $\ast $-Jordan ring isomorphism and (ii) if $\mathcal{B}$ is prime and $\phi (1_{\mathcal{A}})$ is a projection of $\mathcal{B},$ then $\Phi $ is either a $\ast $-ring isomorphism or an $\ast $-ring anti-isomorphism.
\end{corollary}

\begin{corollary} Let $\mathcal{A}$ and $\mathcal{B}$ be two unital complex $\ast $-algebras with $1_{\mathcal{A}}$ and $1_{\mathcal{B}}$ their multiplicative identities, respectively, and such that $\mathcal{A}$ is prime and has a nontrivial projection. Let $\Phi :\mathcal{A}\rightarrow \mathcal{B}$ be a bijective mapping preserving triple product $a\filledsquare _{\eta }b\filledsquare _{\nu }c$ (resp., preserving mixed product $a\filledsquare _{\eta }b\circ _{\nu }c$), where $\eta $ and $\nu $ are nonzero complex numbers, such that $\Phi (1_{\mathcal{A}})$ is a projection and satisfying at least one of the following conditions: (i) ($\overline{\eta }\neq -1$ and $\nu \neq \pm 1$) and $\Phi ((\overline{\eta }+1)(1-\nu )a)=(\overline{\eta }+1)(1-\nu )\Phi (a),$ for all element $a\in \mathcal{A},$ or (ii) ($\overline{\eta }\neq \pm 1$ and $\nu \neq -1$) and $\Phi ((\overline{\eta }-1)(1+\nu )a)=(\overline{\eta }-1)(1+\nu )\Phi (a),$ for all element $a\in \mathcal{A}$ (resp., (i) ($\eta \neq -1$ and $|\nu |\neq 1$) and $\Phi ((\eta +1)(1-\nu )a)=(\eta +1)(1-\nu )\Phi (a),$ for all element $a\in \mathcal{A},$ or (ii) ($|\eta |\neq 1$ and $\nu \neq -1$) and $\Phi ((\eta -1)(1+\nu )a)=(\eta -1)(1+\nu )\Phi (a),$ for all element $a\in \mathcal{A}$). Then, $\Phi $ is a $\ast $-ring isomorphism. 

In particular, If (i) ($\overline{\eta }\neq -1$ and $\nu \neq \pm 1$) and $(\overline{\eta }+1)(1-\nu )$ is a rational number or (ii) ($\overline{\eta }\neq \pm 1$ and $\nu \neq -1$) and $(\overline{\eta }-1)(1+\nu )$ is a rational number (resp., (i) ($\eta \neq -1$ and $|\nu |\neq 1$) and $(\eta +1)(1-\nu )$ is a rational number or (ii) ($|\eta |\neq 1$ and $\nu \neq -1$) and $(\eta -1)(1+\nu )$ is a rational number), then $\Phi $ is a $\ast $-ring isomorphism.
\end{corollary}

\end{document}